\documentclass[11pt,leqno,a4paper]{amsart}
\usepackage[top=2.2cm,bottom=2.2cm,right=2cm,left=2cm]{geometry}
\usepackage{amssymb,amsmath,amscd,graphicx,fontenc,amsthm,mathrsfs, mathtools}
\usepackage[frame, cmtip, arrow, matrix, line, graph, curve]{xy}
\usepackage{graphpap, color}
\usepackage[mathscr]{eucal}
\usepackage{ifthen}
\usepackage{relsize}

\usepackage{multirow}

\usepackage{fancyhdr}
\usepackage{indentfirst}
\usepackage{paralist}
\usepackage{pstricks}
\usepackage{bbm}
\usepackage{tikz}

\newtheorem{theorem}{Theorem}[section]

\newtheorem{proposition}[theorem]{Proposition}

\theoremstyle{remark}

\numberwithin{equation}{section}
\newcommand{\EE}{{\mathbb{E}}}
\newcommand{\NN}{{\mathbb{N}}}
\newcommand{\ZZ}{{\mathbb{Z}}}

\newcommand{\RR}{{\mathbb{R}}}

\renewcommand{\SS}{{\mathbb{S}}}

\DeclareMathOperator{\rk}{rk}

\DeclareMathOperator{\vol}{vol}

\newcommand{\diag}{\mathrm{diag}}

\newcommand{\tp}[1]{^\mathrm{t}{#1}}

\newcommand{\SL}{\mathrm{SL}}

\newcommand{\Id}{\mathrm{Id}}

\newcommand{\vx}{\mathbf x}

\newcommand{\vw}{\mathbf w}
\newcommand{\vv}{\mathbf v}

\providecommand{\ve}{\mathbf{ e}}

\newcommand{\vm}{\mathbf m}

\newcommand{\vk}{\mathbf k}

\newcommand{\Siegel}[2]{\mathcal S_{#1}(#2)}
\newcommand{\PSiegel}[2]{\widehat {\mathcal S}_{#1}(#2)}

\newcommand{\hatcal}[1]{\widehat{\mathcal{#1}}}

\renewcommand{\varpi}{\pi}

\newcommand{\origin}{O}

\definecolor{cmd}{rgb}{1.0, 0.35, 0.21}

\begin{document}
\title[Distribution of Primitive Lattice Points in Large Dimensions]{Distribution of Primitive Lattice Points in Large Dimensions}

\author{Jiyoung Han}


\maketitle
\begin{abstract}
We investigate the asymptotic behavior of the distribution of primitive lattice points in a symmetric Borel set $S_d\subset\RR^d$ as $d$ goes to infinity, under certain volume conditions on $S_d$. Our main technique involves exploring higher moment formulas for the primitive Siegel transform.
We first demonstrate that if the volume of $S_d$ remains fixed for all $d\in \NN$, then the distribution of the half the number of primitive lattice points in $S_d$ converges, in distribution, to the Poisson distribution of mean $\frac 1 2$. 
Furthermore, if the volume of $S_d$ goes to infinity subexponentially as $d$ approaches infinity, the normalized distribution of the half the number of primitive lattice points in $S_d$ converges, in distribution, to the normal distribution $\mathcal N(0,1)$.
We also extend these results to the setting of stochastic processes. 
This work is motivated by the contributions of Rogers \cite{Rogers55-2}, Södergren \cite{Sod2011} and Strömbergsson and Södergren \cite{StSo2019}.
\end{abstract}



\section{Introduction}\label{Section: Introduction}

For $d\ge 2$, one can understand $X_d=\SL_d(\RR)/\SL_d(\ZZ)$ as the space of unimodular lattices in $\RR^d$ via the map $g\SL_d(\ZZ)$ to $g\ZZ^d$, and let $\mu_d$ be the $\SL_d(\RR)$-invariant probability measure on $X_d$. Under this identification, one can define the \emph{Siegel transform}  
\begin{equation}\label{R Siegel}
\widetilde {f}(g\ZZ^d)
=\sum_{\vv\in \ZZ^d-\{\origin\}} f(g\vv),\quad \forall g\ZZ^d \in X_d
\end{equation}
for a bounded and compactly supported function $f$ on $\RR^d$. 
When we take $f$ as the indicator function of a Borel set $S\subseteq \RR^d$, the quantity $\widetilde{f}(g\ZZ^d)$ stands for the number of nontrivial lattice points of $g\ZZ^d$ contained in $A$, and this establishes a connection between the lattice-counting problems, geometry of numbers, and homogeneous dynamics \cite{Sch60, EMM98, EMM05, AM09, MM11, AM18, KY20,  AGY21, HLM}, see also \cite{HLM17, Han22, GH22, SKim24} for S-arithmetic and adelic settings, \cite{Fairchild2021, KS21, KY21, GKY22, SKim22, KY23, BFC} for other Siegel transforms on various homogeneous spaces.

Siegel's famous integral formula \cite{Sie98} says that the mean of $\widetilde{f}$ on $X_d$ with the measure $\mu_d$ is equal to the integral of $f$ with the usual Lebesgue measure on $\RR^d$. In \cite{Rogers55}, Rogers presented higher moment formulas for the Siegel transform (see also \cite{Schmidt57}). Using this result, he established that the asymptotic behavior of the number of lattice points in a Borel set $S_d \subseteq \RR^d$ with $\vol(S_d)=V$, for a fixed positive number $V$, is Poissonian.
In 2011, S$\ddot{\text{o}}$dergren \cite{Sod2011} developed this result of Rogers to the setting of stochastic processes. When the volume $\vol(S_d)$ diverges subexponenetially to infinity as the dimension $d$ increases, S$\ddot{\text{o}}$dergren and Str$\ddot{\text{o}}$mbergsson \cite{StSo2019} accomplished that the asymptotic behavior converges in distribution to normal distribution, based on the idea that the Poisson distribution of large mean resembles the normal distribution.
Recently, Alam, Ghosh and the author \cite{AGH2022} derived higher moment formulas (for rank $\ge 3$) in both the affine and the congruence cases, thereby facilitating analogues of the aforementioned applications.

In this article, our aim is to delve into higher moment formulas for the primitive Siegel transform, exploring their potential applications, where  the transform is defined as
\begin{equation}\label{P Siegel}
\widehat {f}(g\ZZ^d)
=\sum_{\vv\in P(\ZZ^d)} f(g\vv),\quad \forall g\ZZ^d \in X_d
\end{equation}
for a bounded and compactly supported function $f$. Here, $P(\ZZ^d)$ is the set of primitive integer vectors, i.e., the set of integer vectors $\vv$ satisfying that $\RR\vv\cap \ZZ^d=\ZZ\vv$. More generally, we will denote by $P(\Lambda)$ the set of elements $\vv$ of the lattice $\Lambda$ satisfying that $\RR\vv \cap \Lambda=\ZZ\vv$. It is well known that $P(\ZZ^d)=\SL_d(\ZZ).\ve_1$, where $\ve_1={\tp{(1, 0, \ldots, 0)}}\in \RR^d$ and $P(g\ZZ^d)=gP(\ZZ^d)$ for any $g\in \SL_d(\RR)$.

It appears from \cite{Sie98} and \cite{Rogers55} that such formulas have a profound connection to the Riemann zeta function $\zeta(d)$: 
\[\begin{gathered}
\int_{X_d} \widehat{f}(g\ZZ^d) d\mu_d(g)= \frac 1 {\zeta(d)} \int_{\RR^d} f d\vv
\quad\text{for}\;d\ge 2;\\
\int_{X_d} \widehat{f}(g\ZZ^d)^2 d\mu_d(g)=\Big(\frac 1 {\zeta(d)} \int_{\RR^d} f d\vv\Big)^2+ \frac 1 {\zeta(d)} \int_{\RR^d} f(\vv)f(\vv)+f(\vv)f(-\vv)d\vv
\quad\text{for}\;d\ge 3. 
\end{gathered}\]
See also \cite{Sch60} and \cite{FH}, respectively, for the second moment formula of the primitive Siegel transform on the real space and the $S$-arithmetic space, respectively, of dimension $2$.

For higher ranks cases, even the formula for the third moment $\int_{X_d} \widehat{f}(g\ZZ^d)^3 d\mu_d(g)$ of the primitive Siegel transform remains unknown, despite the ease of accomplishing its integrability when $d\ge 4$ from the Riesz--Kakutani--Markov representation theorem. 
Therefore, the main goal of this article is to address whether we can still achieve such applications involving higher moment formulas for Siegel transforms without explicitly describing the exact formulas. This question will be answered through following theorems.

\begin{theorem}\label{Poisson distribution} For a given $V>0$, consider a sequence $\{S_d\}_{d\in \NN}$ of measurable sets $S_d\subseteq \RR^d$ such that $S_d=-S_d$ and $\vol(S_d)=V$.
If we let
\[
\widehat{W}_d= \frac 1 2 \#\left(P(\Lambda) \cap S_d\right),
\]
where $\Lambda$ is randomly chosen in $(X_d, \mu_d)$, then $\widehat{W}_d$ converges to the Poisson distribution with mean $V/2$ in distribution.
\end{theorem}

Recall that \emph{a star-shaped set} $S\subseteq \RR^d$ centered at the origin is given by
\[
S=\{\vv\in \RR^d : \vv < \rho(\vv/\|\vv\|)\}
\] 
for some continuous positive function $\rho: \SS^{d-1}\rightarrow \RR_{>0}$.
For a star-shaped set $S\subseteq \RR^d$ and a non-negative number $t\in \RR_{\ge 0}$, define the dilate of $S$ by $t$ as
\[
tS=\{\vv \in \RR^d : \vv/t\in S\},
\] 
and conventionally put $0S:=\{\origin\}$.


\begin{theorem}\label{Poisson process} Let $\{S_d\}_{d\in \NN}$ be a sequence of measurable star-shaped sets $S_d\subseteq \RR^d$ centered at the origin with $S_d=-S_d$ and $\vol(S_d)=1$.
For $t\in \RR_{\ge 0}$, define
\[
\widehat{W}_d(t)=\frac 1 2 \#\left(P(\Lambda) \cap t^{1/d} S_d\right),
\]
where $\Lambda$ is randomly chosen is $(X_d, \mu_d)$.
Then the stochastic process $\{\widehat{W}_d(t):t\in \RR_{\ge 0}\}$ converges weakly to the Poisson point process on $\RR_{\ge 0}$ with intensity $1/2$. 
\end{theorem}

We remark that the result of Theorem~\ref{Poisson process} also holds for the following setting without any further modification of the proof of the theorem: For each $d\in \NN$, consider the increasing family $\mathcal S_d=\{S^{(d)}_t\subseteq \RR^d: t\in \RR_{\ge 0}\}$ of measurable sets such that $\vol(S^{(d)}_t)=t$ and $S^{(d)}_t=-S^{(d)}_t$, and define
\[
W'_d(t)=\frac 1 2 \#\left(P(\Lambda) \cap S^{(d)}_t \right).
\]
It follows that $\{W'_d(t):t\in \RR_{\ge 0}\}$ converges weakly to the Poisson point process on $\RR_{\ge 0}$ with intensity $1/2$.

\vspace{0.1in}
Let $\phi:\NN\rightarrow \RR_{>0}$ be a function for which
\begin{equation}\label{condition for phi}
\lim_{d\rightarrow \infty} \phi(d)=\infty
\quad\text{and}\quad
\phi(d)=O_{\varepsilon}(e^{\varepsilon d}),\;\forall \varepsilon>0.
\end{equation}

\begin{theorem}\label{Normal distribution}
For each $d\in \NN$, let $S_d\subseteq \RR^d$ be a measurable set such that $S_d=-S_d$ and $\vol(S_d)=\phi(d)$. Define
\[
\widehat Z_d=\frac {\# (P(\Lambda) \cap S_d) - \phi(d)/\zeta(d)} {\sqrt{2\phi(d)/\zeta(d)}},
\]
where $\Lambda$ is randomly chosen in $(X_d, \mu_d)$. It holds that
\[
\widehat Z_d\rightarrow \mathcal N(0,1)\;\text{as}\; d\rightarrow \infty
\]
in distribution. Here, $\mathcal N(0,1)$ is the standard normal distribution.
\end{theorem}

\begin{theorem}\label{Brownian motion}
For each $d\in \NN$, let $S_d\subseteq \RR^d$ be a star-shaped set such that $S_d=-S_d$ and $\vol(S_d)=\phi(d)$.
For $t\in [0,1]$, define
\[
\widehat Z_d(t)=\frac {\#\left(P(\Lambda) \cap t^{1/d}S_d\right)- t\phi(d)/\zeta(d)}{\sqrt{2\phi(d)/\zeta(d)}},
\]
where $\Lambda$ is randomly chosen in $(X_d, \mu_d)$. Then $\widehat Z_d(t)$ converges to one-dimensional Brownian motion in distribution as $d$ goes to infinity.
\end{theorem}

\subsection*{Organization}
In Section~\ref{Section: Preliminaries}, we briefly review Rogers' higher moment formulas for the Siegel transform as defined in \eqref{R Siegel}, along with the necessary properties for the remainder of this article.
In Section~\ref{Section: Convergence to Poisson Distribution}, to prove Theorem~\ref{Poisson distribution} and Theorem~\ref{Poisson process}, we demonstrate that the matrices contributing to the \emph{main term} also appear in the integral formula for the primitive Siegel transform, while the sum of integrals related to the rest of matrices vanishes as $d$ goes to infinity.
For Theorem~\ref{Normal distribution} and Theorem~\ref{Brownian motion}, we further conduct an analysis of the integral formulas for functions defined by the primitive Siegel transform, normalized by their means. For this, in Section~\ref{New Moment Formulas}, we derive \emph{new moment formulas} for these normalized functions, presenting an analogue to \cite[Theorem 2.3]{StSo2019}.

\subsection*{Acknowledgement}
I would like to thank Anish Ghosh for his encouragement and valuable discussions. I am also grateful to Evan O'Dorney for identifying numerous typos and informalities. This project is  supported by a KIAS Individual Grant MG088401 at Korea Institute for Advanced Study.

\section{Preliminaries}\label{Section: Preliminaries}
For a bounded and compactly supported function $F:(\RR^d)^k\rightarrow \RR$, define
\begin{equation}
\Siegel{k}{F}(g\ZZ^d)
=\sum_{\vv_1, \ldots, \vv_k\in \ZZ^d-\{0\}} F(g\vv_1, \ldots, g\vv_k),\quad \forall g\in \SL_d(\RR).
\end{equation}

The following theorem was firstly introduced by Rogers \cite{Rogers55}, and proved by Schmidt \cite{Schmidt57}. The theorem was generalized to the S-arithmetic case by the author \cite{Han22} with a different argument from \cite{Schmidt57}, even in the real case. Let us follow notations in \cite{Han22}.
\begin{theorem}\label{Rogers higher moment formula}
For each $1\le k\le d-1$, $q\in \NN$ and $1\le r \le k$, define $\mathcal D^k_{r,q}$ be the set of $r\times k$ matrices $D$ with integral coefficients for which there are $J_D:=\{1=j_1<j_2<\cdots<j_r\le k\}$ such that
\begin{enumerate}
\item each $[D]^j$ is nonzero;
\item $\left([D]^{j_1}, [D]^{j_2}, \ldots, [D]^{j_r}\right)=qI_r$;
\item $D_{ij}=0$ for $1\le i\le r$ and $1\le j<j_i$;
\item $\gcd(D_{ij})=1$,
\end{enumerate}
where $[D]^j$ and $D_{ij}$, respectively, are the $j$-th column and the $(i,j)$-entry, respectively, of $D$.
Then
\[
\int_{X_d} \Siegel{k}{F}(g\ZZ^d)^k d\mu_d(g)
=\sum_{r=1}^k \sum_{q\in \NN} \sum_{D\in \mathcal D^{k}_{r,q}}
c_D \int_{(\RR^d)^r} F \left(\frac 1 q (\vv_1, \ldots, \vv_r)D\right) d\vv_1 \cdots d\vv_r,
\]
where
\begin{equation}\label{regular coeff}
c_D=\frac{\#\left\{\vx\in \{0,1,\ldots, q-1\}^r :\vx D/q\in \ZZ^k \right\}^d} {q^{dr}}.
\end{equation}
\end{theorem}

Using the fact that $c_D\le 1/q^d$ for any $D\in \mathcal D^k_{r,q}$, Rogers showed the following proposition.

\begin{proposition}[{\cite[Section 9]{Rogers55-2}}]\label{Rogers 55} 
Define
\begin{equation}\label{R1}
\mathcal R^k_1
=\Bigg\{D\in \bigcup_{\scriptsize \begin{array}{c}
k,q\in \NN\\
1\le r \le k\end{array}} \mathcal D^k_{r,q} :
\begin{array}{l}
\circ \; q \ge 2 \quad\text{or}\\
\circ \; q=1 \;\text{and}\; |D_{ij}|\ge 2\;\text{for some}\; D_{ij} 
\end{array}\Bigg\}.
\end{equation}
Assume that $d$ and $k\in \NN$ satisfy the condition that $d\ge \lfloor k^2/4 \rfloor +3$, where $\lfloor x \rfloor$ is the largest integer less than or equal to $x$.
Let $F=\prod_{j=1}^k I_{A_j}$ be the product of indicator functions of $A_j$, where each $A_j \subseteq \RR^d$ is a Borel set contained in some Borel set $B\subseteq \RR^d$ with $\vol(B)=V$.
Then
\[
\sum_{r=1}^{k-1}\sum_{q\in \NN}\sum_{D\in \mathcal D^k_{r,q}\cap \mathcal R^k_1} c_D\int_{(\RR^d)^r} F\left(\frac 1 q (\vv_1, \ldots, \vv_r)D\right) d\vv_1 \cdots d\vv_r
< 21\cdot 5^{\lfloor k^2/4 \rfloor} 2^{-d} (V+1)^k.
\]
where $c_D$ is a positive constant as in \eqref{regular coeff}.
\end{proposition}

Note that $\Big(\underset{1\le r\le k}{\bigcup}\:\underset{q\in \NN}{\bigcup} \mathcal D^k_{r,q} \Big)- \mathcal R^k_1$ consists of matrices $D$ such that $D_{ij}\in \{0, \pm 1\}$ for $1\le i\le r$ and $1\le j\le k$. In particular, such a matrix $D$ is contained in $\mathcal D^k_{r,1}$ for some $1\le r\le k$.

\begin{proposition}[{\cite[Lemma 7]{Rogers56-2}}]\label{Rogers 56}
Define
\begin{equation}\label{R2}
\mathcal R^k_2
=\left\{D\in \bigcup_{1\le r\le k} \mathcal D^k_{r,1} - \mathcal R^k_1 :
\begin{array}{c}
\text{There is $1\le j\le k$ such that}\\
\text{$[D]^j$ has at least two nontrivial entries}\end{array} \right\},
\end{equation}
where $\mathcal R^k_1$ is the set defined as in \eqref{R1}.

Assume that $k\le d-1$.
Let $F=\prod_{j=1}^k I_{A_j}$ be the product of indicator functions of $A_j$, where each $A_j \subseteq \RR^d$ is a Borel set contained in some Borel set $B\subseteq \RR^d$ with $\vol(B)=V$.
Then
\[
\sum_{r=1}^{k-1}\sum_{D\in \mathcal D^k_{r,1}\cap \mathcal R^k_2} c_D\int_{(\RR^d)^r} F\left(\frac 1 q (\vv_1, \ldots, \vv_r)D\right) d\vv_1 \cdots d\vv_r
< 2\cdot 3^{\lfloor k^2/4 \rfloor} \sqrt{\frac 3 4}^{\:d} (V+1)^k,
\]
where $c_D$ is as in \eqref{regular coeff}.
\end{proposition}

The following proposition is easily induced from \cite[page 312]{Rogers56-2} and \cite[Lemma 3]{Sod2011}. However, for the sake of completeness, let us prove the proposition.
\begin{proposition}\label{correspondonce with partitions}
For each $1\le r \le k$, define
\begin{equation}\begin{split}\label{Main}
\mathcal M^k_r&=\mathcal D^k_{r,1}-\left(\mathcal R^k_1 \cup \mathcal R^k_2\right)\\
&=\left\{D\in \mathcal D^k_{r,1} : \text{For each column, there is a unique nonzero entry which is $\pm 1$}\right\}.
\end{split}\end{equation}
There is a $2^{k-r}$-to-one correspondence between $\mathcal M^k_r$ and the collection $\mathcal P^k_r$ of partitions $P=\{B_1, \ldots, B_r\}$ of $\{1,\ldots, k\}$ with $B_j\neq \emptyset$ for all $j$ via the map
\begin{equation}\label{map to partitions}
D \mapsto \{B_1, \ldots, B_r\},\;\text{where}\;
B_i=\{j: D_{ij}\neq 0\}\;\text{for}\;1\le i\le r.
\end{equation}
\end{proposition}
\begin{proof}
If $\{B_1, \ldots, B_r\}$ is the image of $D$ under the above map, then $\min B_i=j_i$, where $j_i$ is as in Theorem~\ref{Rogers higher moment formula} and $D_{ij_i}=1$ for any $1\le i\le r$. Hence for a given partition $\{B_1, \ldots, B_r\}\in \mathcal P^k_r$ with $1=\min B_1<\min B_2 < \cdots < \min B_r$, the corresponding matrices are of the form
\[
D_{ij}=\left\{\begin{array}{cl}
1, &\text{if }j=j_i;\\
\pm 1, &\text{if }j\in B_i-\{j_i\};\\
0, &\text{otherwise.}\end{array}\right.
\]
Therefore the number of $D\in \mathcal M^k_r$ which maps to $\{B_1, \ldots, B_r\}\in \mathcal P^k_r$ is $2^{k-r}$.
\end{proof}

\section{Convergence to Poisson distribution}\label{Section: Convergence to Poisson Distribution}
\subsection{Incomplete Moment Formulas of higher ranks}

For a bounded and compactly supported function $F:(\RR^d)^k\rightarrow \RR$, define
\begin{equation}\label{primitive Siegel}
\PSiegel{k}{F}(g\ZZ^d)
=\sum_{\vv_1, \ldots, \vv_k\in P(\ZZ^d)}
F(g\vv_1,\cdots, g\vv_k),\quad\forall g\in \SL_d(\RR).
\end{equation}

Applying Riesz--Kakutani--Markov representation theorem, and since $\PSiegel{k}{F}\le \Siegel{k}{F}$, it is easy to obtain the following proposition.

\begin{proposition}\label{Primitive Moment Formula}
For each $1\le k\le d-1$, $q\in \NN$ and $1\le r \le k$, the set $\mathcal D^k_{r,q}$ and for each $D\in \mathcal D^k_{r,q}$, the constant $c_D>0$ be as in Theorem~\ref{Rogers higher moment formula}.
Define the set
\[
\hatcal{D}^k_{r,q}=\left\{D\in \mathcal D^k_{r,q} : \begin{array}{c}
\text{There are } \vw_1, \ldots, \vw_r \in P(\ZZ^d) \text{ such that}\\
\frac 1 q (\vw_1, \ldots, \vw_r)D\in P(\ZZ^d)^k.\end{array}\right\}.
\] 

Then there is $0\le \widehat{c}_D\le c_D$ for each $D\in \hatcal{D}^k_{r,q}$ so that the following holds.
\begin{equation}\label{eq: Primitive Moment Formula}
\int_{X_d} \PSiegel{k}{F}(g\ZZ^d) d\mu_d(g)
=\sum_{r=1}^k\sum_{q\in \NN} \sum_{D\in \hatcal{D}^k_{r,q}}
\widehat{c}_D \int_{(\RR^d)^r} F\left(\frac 1 q (\vv_1, \ldots, \vv_r) D \right) d\vv_1 \cdots d\vv_r.
\end{equation}
\end{proposition}
\begin{proof}
The proof is almost identical with the first step in the proof of Theorem 3.1 in \cite{Han22} (see also the beginning of \cite[Section 3]{Han22}), so that let us provide the rough sketch here.
For each $D\in \hatcal D^k_{r,q}$, define 
\begin{equation}\label{Set Phi}
\widehat{\Phi}_D=\left\{(\vw_1, \ldots, \vw_r)\in P(\ZZ^d)^r : \begin{array}{c}
\rk(\vw_1, \ldots, \vw_r)=r\;\text{and}\\
\frac 1 q (\vw_1, \ldots, \vw_r)D\in P(\ZZ^d)^k \end{array}\right\}.
\end{equation}
It is easy to show that
\[
\left\{(\vv_1, \ldots, \vv_k): \vv_1, \ldots, \vv_k \in P(\ZZ^d)\right\}
=\bigsqcup_{r=1}^k \bigsqcup_{q\in \NN} \bigsqcup_{D\in \hatcal D^k_{r,q}} 
\left\{\frac 1 q (\vw_1, \ldots, \vw_r) D : (\vw_1, \ldots, \vw_r)\in \widehat{\Phi}_D \right\}
\]
and the right hand side of \eqref{eq: Primitive Moment Formula} can be decomposed as
\[
\int_{X_d} \PSiegel{k}{F}(g\ZZ^d) d\mu_d(g) 
=\sum_{r=1}^k\sum_{q\in \NN}\sum_{D\in \hatcal D^k_{r,q}}
\int_{X_d} \sum_{\scriptsize \begin{array}{c}
(\vw_1, \ldots, \vw_r)\\
\in \widehat{\Phi}_D \end{array}}
F\left(\frac 1 q (g\vw_1, \ldots, g\vw_r) D\right) d\mu_d(g).
\]

It follows from Riesz--Kakutani--Markov representation theorem that for each $D\in \hatcal D^k_{r,q}$, there is $\widehat{c}_D>0$ for which
\begin{equation}\label{Riesz}
\int_{X_d} \sum_{\scriptsize \begin{array}{c}
(\vw_1, \ldots, \vw_r)\\
\in \widehat{\Phi}_D \end{array}}
F\left(\frac 1 q (g\vw_1, \ldots, g\vw_r) D\right) d\mu_d(g)
=\widehat{c}_D \int_{(\RR^d)^r} F\left(\frac 1 q (\vw_1, \ldots, \vw_r) D\right) d\vw_1 \ldots d\vw_r,
\end{equation}
which explains the integrals in the left summation in \eqref{eq: Primitive Moment Formula}.
Moreover, since $\widehat \Phi_D \subseteq \Phi_D$, we have that $\widehat c_D \le c_D$.
\end{proof}

It is very difficult to compute $\widehat c_D$ for a general $D\in \hatcal D^k_{r,q}$ when $2\le r\le k-1$, which makes hard to obtain the $k$-th moment formula for the primitive Siegel transform when $k\ge 3$.
However, one can compute constants $\widehat c_D$ for special matrices $D$ by comparing to the constants $c_D$ of the higher moment formula for the Siegel transform defined as in \eqref{R Siegel},
and we will see that these computations are enough to obtain our main theorems in the next subsection.

\begin{proposition}\label{Computation of main coeff}
Recall the definition of $\mathcal M^k_r$ in \eqref{Main}. It follows that $\mathcal M^k_r\subseteq \hatcal{D}^k_{r,1}$ and for $D\in \mathcal M^k_r$,
\[
\widehat{c}_D=\frac 1 {\zeta(d)^r}.
\]
\end{proposition}
\begin{proof}
Since the matrix $D\in \mathcal M^k_r$ sends $(\vv_1, \ldots, \vv_r)$ to $(\vv_1, \pm \vv_{i_2}, \ldots, \pm \vv_{i_{k-1}})$, where $i_2, \ldots, i_{k-1}\in \{1,\ldots, r\}$, it follows that $(\vv_1, \ldots, \vv_r)D\in P(\ZZ^d)^k$ if $\vv_1, \ldots, \vv_r\in P(\ZZ^d)$, hence $\mathcal M^k_r \subseteq \hatcal D^k_{1,r}$.

Let us show that $\widehat{c}_D=1/\zeta(d)^r$. 
Let $\widehat{\Phi}_D$ be the set defined as in \eqref{Set Phi} (with $q=1$) and define 
\[
\Phi_D=\left\{(\vw_1, \ldots, \vw_r)\in (\ZZ^d)^r : \rk(\vw_1, \ldots, \vw_r)=r\;\text{and}\;(\vw_1, \ldots, \vw_r)D\in P(\ZZ^d)^k \right\}.
\]

The following equality is known from the proof of Theorem~\ref{Rogers higher moment formula} (see also \cite[Theorem 3.1]{Han22}).
\[
\int_{X_d} \sum_{\scriptsize \begin{array}{c}
(\vw_1, \ldots, \vw_r)\\
\in \Phi_D\end{array}} 
F\left((g\vw_1, \ldots, g\vw_r)D\right) d\mu_d(g)
=\int_{(\RR^d)^r} F\left((\vv_1, \ldots, \vv_r)D\right) d\vv_1 \cdots d\vv_r.
\]
In other words, it holds that $c_D=1$ for any $D\in \mathcal M^k_r$.

It is obvious that 
$\Phi_D=\bigcup_{\vm\in \NN^r}
\left\{(m_1\vw_1, \ldots, m_r\vw_r) : (\vw_1, \ldots, \vw_r)\in \widehat{\Phi}_D \right\}$, where $\vm=(m_1, \ldots, m_r)$.
For any bounded and compactly supported function $F$ on $(\RR^d)^k$, and for each $\vm\in \NN^r$, define
\[
F_\vm(\vv_1, \ldots, \vv_r)=F(m_1 \vv_1, \ldots, m_r \vv_r).
\]
It follows that
\[\begin{split}
&\int_{(\RR^d)^r} F\left((\vv_1, \ldots, \vv_r)D\right) d\vv_1 \cdots d\vv_r
=\int_{X_d} \sum_{\scriptsize \begin{array}{c}
(\vw_1, \ldots, \vw_r)\\
\in \Phi_D\end{array}} 
F\left((g\vw_1, \ldots, g\vw_r)D\right) d\mu_d(g)\\
&=\sum_{\vm \in \NN^r} \int_{X_d} \sum_{\scriptsize \begin{array}{c}
(\vw_1, \ldots, \vw_r)\\
\in \widehat{\Phi}_D\end{array}} 
F\left( (gm_1\vw_1, \ldots, gm_r\vw_r)D\right) d\mu_d(g)\\
&=\sum_{\vm\in \NN^r} \int_{X_d} \sum_{\scriptsize \begin{array}{c}
(\vw_1, \ldots, \vw_r)\\
\in \widehat{\Phi}_D\end{array}} F_\vm \left((g\vw_1, \ldots, g\vw_r)D\right) d\mu_d(g)
=\sum_{\vm\in \NN^r} \widehat{c}_D \int_{(\RR^d)^r}F_\vm \left( (\vv_1, \ldots, \vv_r)D\right)d\vv_1 \cdots d\vv_r\\
&=\sum_{\vm\in \NN^r} \widehat{c}_D \frac 1 {m_1^d \cdots m_r^d}
\int_{(\RR^d)^r} F\left((\vv_1, \ldots, \vv_r)D\right) d\vv_1 \cdots d\vv_r,
\end{split}\]
hence $1=\widehat c_D\cdot \sum_{\vm\in \NN^r} 1/m_1^d\cdots m_r^d=\widehat c_D \cdot \zeta(d)^r$, i.e., $\widehat c_D=1/\zeta(d)^r$.
\end{proof}

\subsection{Proof of Theorem~\ref{Poisson distribution} and \ref{Poisson process}}


For each $\ell\in \NN$ and any $\vk=(k_1,\ldots, k_\ell)\in \NN^\ell$, fix $0\le t_1< t_2 < \cdots < t_\ell$.
Let $k=\sum_{j=1}^\ell k_j$. Let $\mathcal P^k_r$ be the collection of partitions $P=\{B_1, \ldots, B_r\}$ of $\{1, \ldots, k\}$ with $B_j\neq \emptyset$ for all $j$ and $\mathcal P^k=\bigcup_{r=1}^k \mathcal P^k_r$.
For any $B\subseteq \{1, \ldots, k\}$, set $t_{B}=\min\{ t_\beta : \beta \in B\}$.

Recall that the $\vk$-th moment of $(\widetilde W(t_1), \ldots, \widetilde W(t_\ell))$, where $\{\widetilde{W}(t): t\in \RR_{\ge 0}\}$ is the Poisson point process with intensity $1/2$, is given by
\[
\EE\left(\prod_{j=1}^\ell \widetilde W(t_j)^{k_j}\right)
=\sum_{P\in \mathcal P^k} 2^{-\# P} \prod_{B\in P} t_B
\] 
(see \cite[Equation (10)]{Sod2011} for instance).

\begin{proof}[Proof of Theorem~\ref{Poisson process}]
Let $\vk\in \NN^\ell$, $k\in \NN$ and $0\le t_1 < \cdots < t_\ell$ be as above. 
It suffices to show that 
\[
(\widehat{W}_d(t_1), \ldots, \widehat{W}_d(t_\ell))\rightarrow (\widetilde W(t_1), \ldots, \widetilde W(t_\ell))
\quad\text{as}\quad d\rightarrow \infty
\]
in distribution (see \cite[Theorem 12.6]{Billingsley} for instance).
Using the (multi-dimensional) method of moments, it is enough to show that
\[
\lim_{d\rightarrow \infty} \EE\left(\prod_{j=1}^\ell \widehat{W}_d(t_j)^{k_j}\right)
=\sum_{P\in \mathcal P^k} 2^{-\# P} \prod_{B\in P} t_B.
\]

We will prove the above formula by two steps. 
First, we claim that
\begin{equation}\label{eq 1: Poisson process}
\EE\left(\prod_{j=1}^\ell \widehat{W}_d(t_j)^{k_j}\right)
=\frac 1 {2^k}\sum_{J\subseteq \{1,\ldots, k\}}M_{J} \prod_{j\in J} \frac {t_j}{\zeta(d)}
+O\left(5^{\lfloor k^2/4\rfloor}2^{-d}t_{\ell}^k + 3^{\lfloor k^2/4\rfloor}\sqrt{3/4}^dt_{\ell}^k\right),
\end{equation}
where $M_J$ is the number of $D\in \mathcal M^k_{(\# J)}$ for which $J_D=J$.

Recall the definitions of $\mathcal R^k_1$, $\mathcal R^k_2$ and $\mathcal M^k_r$ in Section~\ref{Section: Preliminaries}.
For each $j\in \{1, \ldots, \ell\}$, let $f_{j}$ be the half of the indicator function of $t_j^{1/d}S_d\subseteq \RR^d$. By Theorem~\ref{Primitive Moment Formula},
\[\begin{split}
\EE\left(\prod_{j=1}^\ell \widehat{W}_d(t_j)^{k_j}\right)
&=\int_{X_d} \prod_{j=1}^\ell \widehat f_{j}(\Lambda)^{k_j} d\mu_d(\Lambda)
=\int_{X_d} \PSiegel{k}{\prod_{j=1}^\ell f_{j}^{k_j}}(\Lambda) d\mu_d(\Lambda)\\
&=\sum_{r=1}^k \sum_{q\in \NN} \sum_{D\in \hatcal D^k_{r,q}} \widehat c_D
\int_{(\RR^d)^r} \prod_{j=1}^\ell f_{j}^{k_j} \left(\frac 1 q (\vv_1, \ldots, \vv_r)D\right) d\vv_1 \cdots d\vv_r\\
&=\sum_{r=1}^k \sum_{D\in \mathcal M^k_r} \frac 1 {\zeta(d)^r}
\int_{(\RR^d)^r} \prod_{j=1}^\ell f_{j}^{k_j} \left( (\vv_1, \ldots, \vv_r)D\right)d\vv_1 \cdots d\vv_r\\
&\hspace{0.4in}+\sum_{r=1}^k \sum_{q\in \NN} \sum_{D\in \hatcal D^k_{r,q}\cap (\mathcal R^k_1 \cup \mathcal R^k_2)}
\widehat c_D\int_{(\RR^d)^r} \prod_{j=1}^\ell f_{j}^{k_j} \left(\frac 1 q (\vv_1, \ldots, \vv_r)D\right) d\vv_1 \cdots d\vv_r.
\end{split}\]

By Theorem~\ref{Rogers 55} and Theorem~\ref{Rogers 56}, since $0\le \widehat c_D\le c_D$ and $f_j$'s are non-negative,
\[
\sum_{r=1}^k \sum_{q\in \NN} 
\hspace{-0.1in}\sum_{\scriptsize \begin{array}{c}
D\in \\
\hatcal D^k_{r,q}\cap (\mathcal R^k_1 \cup \mathcal R^k_2)\end{array}}
\hspace{-0.25in}\widehat c_D\int_{(\RR^d)^r} \prod_{j=1}^\ell f_{j}^{k_j} \left(\frac 1 q (\vv_1, \ldots, \vv_r)D\right) d\vv_1 \cdots d\vv_r
=O\left(5^{\lfloor k^2/4\rfloor}2^{-d}t_{\ell}^k + 3^{\lfloor k^2/4\rfloor}\sqrt{3/4}^dt_{\ell}^k\right)
\]
which will vanish as $d$ goes to infinity.

By Theorem~\ref{Computation of main coeff}, since $\# J_D=r$ for each $D\in \mathcal M^k_r$,
\[\begin{split}
&\sum_{r=1}^k \sum_{D\in \mathcal M^k_r} \frac 1 {\zeta(d)^r}
\int_{(\RR^d)^r} \prod_{j=1}^\ell f_{j}^{k_j} \left( (\vv_1, \ldots, \vv_r)D\right)d\vv_1 \cdots d\vv_r\\
&\hspace{0in}=\sum_{r=1}^k \sum_{D\in \mathcal M^k_r} \frac 1 {\zeta(d)^r} \frac 1 {2^k} \prod_{j\in J_D} t_j
= \frac 1 {2^k} \sum_{r=1}^k \sum_{D\in \mathcal M_r^k} \prod_{j\in J_D} \frac {t_j} {\zeta(d)}
=\frac 1 {2^k}\sum_{J\subseteq \{1,\ldots, k\}}M_{J} \prod_{j\in J} \frac {t_j}{\zeta(d)}
\end{split}\]
by the definition of $M_J$, which shows \eqref{eq 1: Poisson process}.

Next, let us show that
\begin{equation}\label{eq 2: Poisson process}
\lim_{d\rightarrow \infty} \frac 1 {2^k}\sum_{J\subseteq \{1, \ldots, k\}} M_J\prod_{j\in J} \frac {t_j}{\zeta(d)}=\sum_{P\in \mathcal P^k} 2^{-\# P} \prod_{B\in P} t_B.
\end{equation}

Recall the notation $P=\{B_1, \ldots, B_r\}$ for an element of $\mathcal P^k_r$, where $1\le r\le k$. For each $J\subseteq \{1, \ldots, k\}$ such that $\# J=r$,
\[
\left\{D\in \mathcal M^k_r : J_D=J\right\}
=\bigcup_{\scriptsize \begin{array}{c}
P\in \mathcal P^k_r\\
\{\min B_i\}=J\end{array}}
\hspace{-0.15in}\left\{D\in \mathcal M^k_r \text{ maps to $P$ as in \eqref{map to partitions}}\right\}
\]
and it holds that $\# \left\{D\in \mathcal M^k_r \text{ maps to $P$ as in \eqref{map to partitions}}\right\}=2^{k-r}$ by Proposition~\ref{correspondonce with partitions}.
It follows that since $\lim_{d\rightarrow\infty}\zeta(d)=1$,
\[\begin{split}
\frac 1 {2^k} \sum_{J\subseteq \{1, \ldots, k\}} M_J \prod_{j\in J} \frac {t_j} {\zeta(d)}
&=\frac 1 {2^k} \sum_{r=1}^k \sum_{\scriptsize \begin{array}{c}
J\\
\# J=r\end{array}} M_J \prod_{j\in J} \frac {t_j} {\zeta(d)}
=\frac 1 {2^k} \sum_{\scriptsize \begin{array}{c}
P\in \mathcal P^k_r\\
\{\min B_i\}=J\end{array}} 2^{k-r} \prod_{B\in P} \frac {t_B}{\zeta(d)}\\
&=\sum_{P\in \mathcal P} 2^{-\# P} \prod_{B\in P} \frac {t_B} {\zeta(d)} 
\underset{d\rightarrow \infty}{\longrightarrow} \sum_{P\in \mathcal P^k} 2^{-\# P} \prod_{B\in P} t_B. 
\end{split}\]

Therefore the result follows from \eqref{eq 1: Poisson process} and \eqref{eq 2: Poisson process}.
\end{proof}

\begin{proof}[Proof of Theorem~\ref{Poisson distribution}]
The theorem follows directly from the proof of Theorem~\ref{Poisson process} with $\ell=1$.
For a given $V\in \RR_{\ge 0}$ and $k\in \NN$, we obtain that 
\[\begin{split}
\lim_{d\rightarrow\infty} \EE\left(\widehat{W}_d^k\right)
&=\lim_{d\rightarrow \infty} \frac 1 {2^k} \sum_{J\subseteq\{1, \ldots, k\}} M_J\left(\frac V {\zeta(d)}\right)^{\# J}
=\lim_{d\rightarrow \infty}\sum_{r=1}^k (\# \mathcal P^k_r) \left( \frac V {2\zeta(d)}\right)^r\\
&=\lim_{d\rightarrow \infty}\sum_{r=0}^\infty \frac {r^k} {r!} \left( \frac V {2\zeta(d)}\right)^r
=\sum_{r=0}^\infty \frac {r^k} {r!} \left( \frac V 2\right)^r,
\end{split}\]
which is the $k$-th moment of Poisson distribution with mean $V/2$.
Here, the second last equality is exactly \cite[Equation (10)]{Rogers56-2}.
\end{proof}

\section{Convergence to normal distribution}\label{Section: Convergence to Normal Distribution}
\subsection{New Moment Formulas}\label{New Moment Formulas}
In \cite{StSo2019}, to investigate the moments of $\widehat{Z}_d$ in Theorem~\ref{Normal distribution}, Strömbergsson and Södergren derived more efficient formulas from Rogers' formulas for the Siegel transform, restricting to the space of functions on $\RR^d$ with zero integrals.
For this purpose, they demonstrated that the (partial) sum of integrals in the Rogers' formula over matrices $D$, extracted from a given matrix $D''$ (see \eqref{D and D''} for the relation between $D$ and $D''$), either becomes annihilated entirely or results in only one surviving term, based on the observation that for such $D$ and $D''$, it follows from the definition \eqref{regular coeff} that $c_D=c_{D''}$.

To adopt their strategy, we want to show the property that $\widehat{c}_D=\widehat{c}_{D''}$ under the relation \eqref{D and D''} which will be challenging since now we don't know the computable description such as \eqref{regular coeff}.

\begin{theorem}\label{New Moment Formula}
Let $F:(\RR^d)^k\rightarrow \RR$ be the product of functions, where each of them is of the form $f_j-\frac 1 {\zeta(d)}\int_{\RR^d} f_j d\vv$ for some bounded and compactly supported function $f_j:\RR^d\rightarrow \RR_{\ge 0}$ for $1\le j\le k$.
Then
\begin{equation}\label{eq 0: New Rogers}
\int_{X_d} \PSiegel{k}{F} (g\ZZ^d) d\mu_d(g)
=\sum_{r=1}^k \sum_{q\in \NN} \sum_{D\in N(\hatcal D^k_{r,q})}
\widehat{c}_D \int_{(\RR^d)^r}
\prod_{j=1}^k f_j\left(\frac 1 q (\vv_1, \ldots, \vv_r)D \right) d\vv_1 \cdots d\vv_r,
\end{equation}
where $N(\hatcal D^k_{r,q})$ is the set of $D\in \hatcal D^k_{r,q}$ for which there are at least two nonzero entries in each row and $\widehat{c}_D$ is defined as in Theorem~\ref{Primitive Moment Formula}.
\end{theorem}
\begin{proof}
Since $\PSiegel{k}{F}(\Lambda)=\prod_{j=1}^k \left(\widehat{f_j}(\Lambda)-\int_{\RR^d} f_j d\vv\right)$ for any $\Lambda\in X_d$,
\begin{equation}\label{eq 5: New Rogers}\begin{split}
&\int_{X_d} \PSiegel{k}{F} (g\ZZ^d) d\mu(g)
=\sum_{A\subseteq \{1,\ldots, k\}}(-1)^{a}\prod_{j\in A}\left(\frac 1 {\zeta(d)}\int_{\RR^d} f_j d\vv \right)\times\\
&\hspace{1.8in}
\left(\sum_{r=1}^{k-a} \sum_{q\in \NN} \sum_{D\in \hatcal D^{k-a}_{r,q}}
\widehat c_D \int_{(\RR^d)^r} \prod_{j\in A^c} f_j \left(\frac 1 q (\vv_1, \ldots, \vv_r)D \right) d\vv_1 \cdots d\vv_r \right),
\end{split}\end{equation}
where $a=\# A$ and $A^c=\{1, \ldots, k\}-A$. It is not hard to show that for a given set $A$ and a matrix $D\in \hatcal D^{k-a}_{r,q}$, there is a unique $D''\in \mathcal D^k_{r+a,q}$ such that
\begin{equation}\label{D and D''}\begin{split}
&\prod_{j\in A}\int_{\RR^d} f_j d\vv \int_{(\RR^d)^r} \prod_{j\in A^c} f_j \left(\frac 1 q (\vv_1, \ldots, \vv_r)D\right) d\vv_1 \cdots d\vv_r\\
&\hspace{1in}=\int_{(\RR^d)^{r+a}} \prod_{j=1}^k f_j \left(\frac 1 q (\vv_1, \ldots, \vv_r) D''\right) d\vv_1 \cdots d\vv_{r+a}.
\end{split}\end{equation}
Precisely, $D''$ is defined as follow. Denote $A=\{j_1, \ldots, j_a\}$. Put $\{j'_1, \ldots, j'_{k-a}\}=\{1,\ldots, k\}-A$ and $\{i'_1, \ldots, i'_r\}=\{1, \ldots, r+a\}-A$ (notice that  $A\subseteq\{1, \ldots, r+a\}$). Then
\[
D''_{ij}=\left\{\begin{array}{cl}
q, &\text{if } i=j=j_s,\;\text{for some}\; s=1, \ldots, a;\\
D_{st}, &\text{if } i=i'_s,\; j=j'_{t};\\
0, &\text{otherwise.}
\end{array}\right.
\]

We claim that $D''\in \hatcal D^k_{r+a,q}$ and moreover,
\[
\widehat c_{D''}=\frac{\widehat c_{D}}{\zeta(d)^a}.
\]
After reordering the set $\{f_1, \ldots, f_k\}$ if necessary, we may assume that $A=\{1,\ldots, a\}$ hence $D''$ is the block diagonal matrix
\[
D''=\left(\begin{array}{cc}
q\Id_a & \\
& D \end{array}\right).
\]

Pick any $(\vw_{a+1}, \ldots, \vw_{a+r})\in \widehat{\Phi}_D$, where $\widehat{\Phi}_D$ is the set given as in \eqref{Set Phi}. Choose any $\vw_1,\ldots, \vw_a\in P(\ZZ^d)$ such that $\{\vw_1, \ldots, \vw_{a+r}\}$ is linearly independent (which is possible since $a+r\le k<d$). It follows clearly that $\frac 1 q \left(\vw_1, \ldots, \vw_{a+r}\right)D''\in P(\ZZ^d)^k$, which shows that $D''\in \hatcal D^k_{r+a,q}$.

Next, let us show that $\widehat c_{D''}=\widehat c_D/\zeta(d)^a$. For the sake of notational convenience, one can further assume that $I_D=\{1, \ldots, r\}$ so that the first $(a+r)\times (a+r)$ minor of $D''$ is $q\Id_{a+r}$.
Fix any $(\vw_{a+1}, \ldots, \vw_{a+r})\in \widehat{\Phi}_D$.
It follows that
\[
\widehat{\Phi}_{D''}=\left\{(\vw_1, \ldots, \vw_a)\in P(\ZZ^d)^a: \rk(\vw_1, \ldots, \vw_a, \vw_{a+1}, \ldots, \vw_{a+r})=a+r\right\}\times \widehat{\Phi}_D.
\]
Notice that the former set in (R.H.S) is clearly independent to the choice of $(\vw_{a+1}, \ldots, \vw_{a+r})\in \widehat{\Phi}_D$.

We will utilize the sequence of test functions $F_R$ on $(\RR^d)^k$ defined as follows.
\[
F_R(\vv_1, \ldots, \vv_k)=\frac 1 {\vol(B_R(\origin))^{a+r}}
\prod_{j=1}^{a+r} I_{B_R(\origin)} \times \prod_{j=a+r+1}^k I_{B_{cR}(\origin)}.
\]
Here, $B_R(\origin)\subseteq \RR^d$ is the ball of radius $R$ centered at the origin, $I_A$ is the indicator function of a Borel set $A\subseteq \RR^d$, and $c>r\max\{|D_{ij}|,1\}$, where $D=(D_{ij})$. The constant $c$ is chosen so that for any $(\vv_{a+1}, \ldots, \vv_{a+r})\in (\RR^d)^r$, it follows that
\[
(\vv_{a+1}, \ldots, \vv_{a+r})\in \prod_{j=1}^r B_R(\origin)
\;\Leftrightarrow\;
\frac 1 q(\vv_{a+1}, \ldots, \vv_{a+r})D \in \prod_{j=1}^r B_R(\origin)\times \prod_{j=r+1}^{k-a} B_{cR}(\origin).
\]
It follows directly from \eqref{Riesz} and the above property that
\[
\widehat{c}_{D''}
=\lim_{R\rightarrow \infty} \int_{X_d}
\sum_{\scriptsize \begin{array}{c}
(\vv_1, \ldots, \vv_k)\\
\in\frac 1 q \widehat{\Phi}_{D''}D''\end{array}}
F_R(g\vv_1, \ldots, g\vv_k) d\mu_d(g)
\]
(note that the above equality holds for any $R>0$, without taking the limit).
We assert that
\[
\widehat{c}_{D''}=\lim_{R\rightarrow \infty}
\int_{X_d} \sum_{\scriptsize\begin{array}{c}
(\vv_1, \ldots, \vv_k)\\
\in P(\ZZ^d)^a\times \frac 1 q \widehat{\Phi}_{D}D
\end{array}}F_R(g\vv_1, \ldots, g\vv_k) d\mu_d(g).
\]
Indeed, the difference between the set $\widehat{\Phi}_{D''}$ and $P(\ZZ^d)^a\times \widehat{\Phi}_D$ is contained in 
\[
\Psi:=\left\{(\vv_1, \ldots, \vv_{a+r})\in (\ZZ^d)^{a+r}: \rk(\vv_1, \ldots, \vv_{a+r})\le a+r-1\right\}.
\]
Since the growth of $\#\left(g\Psi \cap \prod_{j=1}^{a+r} B_R(\origin)\right)$ is $O_g(R^{d(a+r)-1})$ (which is $o_g(\vol(B_R(\origin))^{a+r})$) for any $g\in \SL_d(\RR)$ as $R\rightarrow \infty$ (see \cite[Section 4, Lemma 1]{Rogers55}, together with the classical argument with the Margulis $\alpha$-function, we obtain the assertion.

Before proceeding the computation, observe that by \cite[Theorem 1]{Sch60}, it holds that for almost all $g\ZZ^d\in X_d$,
\[
\lim_{R\rightarrow \infty} \frac 1 {\vol(B_R(\origin))} \#\left\{\vv \in P(\ZZ^d) : \|g\vv\|<R\right\}= \frac 1 {\zeta(d)}.
\]
Hence it follows that for almost all $g\ZZ^d\in X_d$,
\begin{equation}\label{eq 2: New Rogers}\begin{split}
&\lim_{R\rightarrow \infty} \frac 1 {\vol(B_R(\origin))^a} \#\left\{(\vv_1, \ldots, \vv_{a}) \in P(\ZZ^d)^a : \|g\vv_i\|\le R\right\}\\
&\hspace{0.5in}=\lim_{R\rightarrow \infty} \left(\frac 1 {\vol(B_R(\origin))} \#\left\{\vv \in P(\ZZ^d) : \|g\vv\|< R\right\}\right)^a
= \frac 1 {\zeta(d)^a}.
\end{split}\end{equation}

Now, applying Lebesgue's dominated convergence theorem and using \eqref{eq 2: New Rogers}, it follows that
\[\begin{split}
\widehat{c}_{D''}&=\int_{X_d}
\Big(\lim_{R\rightarrow \infty}\sum_{\scriptsize \begin{array}{c}
(\vv_1, \ldots, \vv_a)\\
\in P(\ZZ^d)^a
\end{array}}\frac 1 {\vol(B_R(\origin))^a}\prod_{j=1}^a I_{B_R(\origin)} (g\vv_1, \ldots, g\vv_a)\Big)\times\\
&\hspace{0.5in}\Big(\lim_{R\rightarrow \infty} \sum_{\scriptsize \begin{array}{c}
(\vv_{a+1}, \ldots, \vv_{k})\\
\in \frac 1 q\widehat{\Phi}_D D
\end{array}}\frac 1 {\vol(B_R(\origin))^r}\prod_{j=1}^r I_{B_R(\origin)}\times \prod_{j=r+1}^{k-a} I_{B_{cR}(\origin)}(g\vv_{a+1}, \ldots, g\vv_k)\Big)d\mu_d(g)\\
&=\lim_{R\rightarrow \infty} \int_{X_d} \frac 1 {\zeta(d)^a} 
\sum_{\scriptsize \begin{array}{c}
(\vv_{a+1}, \ldots, \vv_{k})\\
\in \frac 1 q\widehat{\Phi}_D D
\end{array}}\frac 1 {\vol(B_R(\origin))^r}\prod_{j=1}^r I_{B_R(\origin)}\times \prod_{j=r+1}^{k-a} I_{B_{cR}(\origin)}(g\vv_{a+1}, \ldots, g\vv_k) d\mu_d(g)\\
&=\frac 1 {\zeta(d)^a} \widehat{c}_D.
\end{split}\]

\vspace{0.1in}
The rest of the proof follows from the similar argument to that of \cite[Theorem 2.4]{StSo2019}, which establishes that if there is a row in $D''\in \hatcal D^k_{r,q}$ with a unique nonzero entry, then the summand of integrals associated to $D''$ is annihilated in (L.H.S) of \eqref{eq 0: New Rogers}; conversely, if every row of $D''$ has at least two nonzero entries, then only possible $A$ is the empty set and the integral corresponding to $D''$ survives.  
\end{proof}

\subsection{Proof of Theorem~\ref{Normal distribution} and ~\ref{Brownian motion}}

Let $\phi: \NN\rightarrow \RR_{> 0}$ be a function such that 
\[
\lim_{d\rightarrow \infty} \phi(d)=\infty
\quad\text{and}\quad
\phi(d)=O_\varepsilon\big(e^{\varepsilon d}\big),\;\forall \varepsilon>0.
\]

For each $d\in \NN$ and $c_1, \cdots, c_\ell>0$, let $S_{d,1}, \ldots, S_{d,\ell}$ be Borel measurable subsets of $\RR^d$ such that $S_{d,j}=-S_{d,j}$ and $\vol(S_{d,j})=c_j\phi(d)$ for all $1\le j\le \ell$. Assume further that $S_{d,j} \cap S_{d,j'}=\emptyset$ if $j\neq j'$. Set
\[
\widehat{Z}_{d,j}:=
\frac {\#(P(g\ZZ^d) \cap S_{d,j})- c_j\phi(d)}{\sqrt{2\phi(d)/\zeta(d)}},
\]
where $g\ZZ^d$ is picked randomly in $(X_d, \mu_d)$.

\begin{proposition}\label{Moment for Brownian motion}
For $k_1, \ldots, k_\ell\in \ZZ_{\ge 0}$, it follows that
\[
\lim_{d\rightarrow \infty}
\EE\left(\widehat Z_{d,1}^{k_1} \cdots \widehat Z_{d,\ell}^{k_\ell} \right)=\left\{\begin{array}{cl}
\prod_{j=1}^\ell \big(c_j^{k_j/2}(k_j-1)!!\big), &\text{if }k_1, \ldots, k_\ell\in 2\NN;\\[0.075in]
0, &\text{otherwise.}
\end{array}\right.
\] 
\end{proposition}

\begin{proof}
Let $f_{d,j}$ be the indicator function of $S_{d,j}$ for each $1\le j\le \ell$. Note that $\frac 1 {\zeta(d)} \int f_{d,j} d\vv= \frac {c_j\phi(d)}{\zeta(d)}$. Let $k=k_1+\cdots+k_\ell$.
By Theorem~\ref{New Moment Formula} and Propositions~\ref{Rogers 55} and \ref{Rogers 56}, it follows that
\[\begin{split}
\EE\left(\widehat Z_{d,1}^{k_1} \cdots \widehat Z_{d,\ell}^{k_\ell} \right)
&=\frac 1 {\sqrt{2\phi(d)/\zeta(d)}^k}
\sum_{r=1}^k \sum_{q\in \NN} \sum_{D\in N(\widehat{\mathcal D}_{r,q}^k)}
\int_{(\RR^d)^r}
\prod_{j=1}^\ell f_{d,j}^{k_j} \left(\frac 1 q (\vv_1, \ldots, \vv_r) D\right) d\vv_1 \cdots d\vv_r\\
&\hspace{-0.4in}=\frac 1 {\sqrt{2\phi(d)/\zeta(d)}^k}
\sum_{r=1}^k \sum_{D\in N(\widehat{\mathcal D}_{r,1}^k)\cap \mathcal M^k_r}\int_{(\RR^d)^r}
\prod_{j=1}^\ell f_{d,j}^{k_j} \left((\vv_1, \ldots, \vv_r) D\right) d\vv_1 \cdots d\vv_r
+O_k\left(\sqrt{\frac 3 4}^d\right).
\end{split}\]
Recall that $\mathcal M^k_r$ is the set of matrices in $\mathcal D^k_{r,1}$ such that for each column, there is a unique nonzero entry which is $\pm 1$.

Since we assume that $\{S_{d,j}\}$ is mutually disjoint, the matrices $D\in N(\widehat{\mathcal D}_{r,q}^k)\cap \mathcal M^k_r$ which has nonzero integral values in the summation above, are block diagonal matrices of the form $\diag(D_1, \ldots, D_\ell)$, where each $D_j$ is an $r_j\times k_j$ matrix such that
\begin{itemize}
\item each column, there is a unique nonzero entry which is $\pm 1$;
\item each row, there are at least two nonzero entries. In particular, it follows that $1\le r_j \le \lfloor \frac {k_j} 2 \rfloor$ for each $1\le j\le \ell$.
\end{itemize}

Hence one can estimate the $(k_1, \ldots, k_\ell)$-moment
\[\begin{split}
&\EE\left(\widehat Z_{d,1}^{k_1} \cdots \widehat Z_{d,\ell}^{k_\ell} \right)\\
&=\prod_{j=1}^\ell \left(\frac 1 {\sqrt{2\phi(d)/\zeta(d)}^{k_j}}
\sum_{r_j=1}^{\lfloor \frac {k_j} 2 \rfloor}
\sum_{D\in N(\widehat{\mathcal D}^{k_j}_{r_j,1}) \cap \mathcal M^{k_j}_{r_j}}
\int_{(\RR^d)^{r_j}} f_{d,j}^{k_j} \left((\vv_1, \ldots, \vv_{r_j})D\right) d\vv_1\cdots d\vv_{r_j}
\right)
+O_k\left(\sqrt{\frac 3 4}^d\right).
\end{split}\]
Since $\# N(\widehat{\mathcal D}^{k_j}_{r_j,1}) \cap \mathcal M^{k_j}_{r_j}=O_{k_j, r_j}(1)$ and $\int_{(\RR^d)^{r_j}} f_{d,j}^{k_j} \left((\vv_1, \ldots, \vv_{r_j})D\right) d\vv_1\cdots d\vv_{r_j}=c_j\phi(d)^{r_j}$, we have that
\[
\frac 1 {\sqrt{2\phi(d)/\zeta(d)}^{k_j}}
\sum_{D\in N(\widehat{\mathcal D}^{k_j}_{r_j,1}) \cap \mathcal M^{k_j}_{r_j}}
\int_{(\RR^d)^{r_j}} f_{d,j}^{k_j} \left((\vv_1, \ldots, \vv_{r_j})D\right) d\vv_1\cdots d\vv_{r_j}
=O_{k_j,r_j} \left(\phi(d)^{r_j- \frac {k_j} 2}\right)
\]
which goes to zero as $d$ goes to infinity except $r_j=k_j/2$.  And if $r_j=k_j/2$, it is easy to show that the above summation is $c_j^{k_j/2}(k_j-1)!!$ by the induction on positive even integers. Therefore, the limit of $\EE\left(\widehat Z_{d,1}^{k_1} \cdots \widehat Z_{d,\ell}^{k_\ell} \right)$ as $d$ goes to infinity does not disappear provided that all $k_1, \ldots, k_\ell$ are even integers and the limit in this case is
\[
\lim_{d\rightarrow \infty}\EE\left(\widehat Z_{d,1}^{k_1} \cdots \widehat Z_{d,\ell}^{k_\ell} \right)=\prod_{j=1}^\ell \big( c_j^{k_j/2} (k_j-1)!!\big).
\] 
\end{proof}

\begin{proof}[Proof of Theorem~\ref{Normal distribution}]
This is the direct consequence of the method of moments and Proposition~\ref{Moment for Brownian motion} by putting $\ell=1$ and $c_1=1$. Note that for each $k\in \NN$, Proposition~\ref{Moment for Brownian motion} says that
\[
\lim_{d\rightarrow \infty} \EE\left(\widehat{Z}_d^k\right)
=\left\{\begin{array}{cl}
(k-1)!!, &\text{if } k\in 2\NN;\\[0.05in]
0, &\text{otherwise,}\end{array}\right.
\]
which equals to the $k$-th moment of the normal distribution $\mathcal N(0,1)$.
\end{proof}

\begin{proof}[Proof of Theorem~\ref{Brownian motion}]
The proof is identical with that of \cite[Theorem 1.6]{StSo2019} using Proposition~\ref{Moment for Brownian motion} instead of \cite[Proposition 4.1]{StSo2019}.
Notice that the tightness, which is achieved from the inequality
\[
\EE \left((\widehat{Z}_d(s)-\widehat{Z}_d(r))^2(\widehat{Z}_d(t)-\widehat{Z}_d(s))^2\right)
\ll \left(\sqrt{t} - \sqrt{r}\right)^2,
\quad \text{for any}\; 0\le r\le s\le t\le 1
\]
is a direct consequence of the inequality \cite[(4.5)]{StSo2019}, since we have that $\bigcup_{u\in \NN} N(\widehat{\mathcal D}^4_{1,u})\cap (\mathcal R^4_1 \cup \mathcal R^4_2)=\emptyset$ (see the last two paragraphs of the proof of \cite[Theorem 1.5 and 1.6]{AGH2022} for details).
\end{proof}



\begin{thebibliography}{1}
%
\bibitem{AGH2022} M. Alam, A. Ghosh and J. Han,
\emph{Higher moment formulae and limiting distributions of lattice points},
J. Inst. Math. Jussieu (online published, doi:10.1017/S147474802300035X)
%
\bibitem{AGY21} M. Alam, A. Ghosh and S. Yu, \emph{Quantitative Diophantine approximation with congruence conditions}, J. Théor. Nombres Bordeaux 33 (2021), no. 1, 261–271.
%
\bibitem{AM09} J. Athreya and G. Margulis, \emph{Logarithm laws for unipotent flows I}, J. Mod. Dyn. 3 (2009), no. 3, 359-378.
%
\bibitem{AM18} J. Athreya and G. Margulis, \emph{Values of random polynomials at integer points}, J. Mod. Dyn. 12 (2018), 9-16.
%
\bibitem{Billingsley} P. Billingsley, 
\emph{Convergence of probability measures} (English summary)
Second edition. Wiley Series in Probability and Statistics: Probability and Statistics. A Wiley-Interscience Publication. John Wiley \& Sons, Inc., New York, 1999. x+277 
%
\bibitem{BFC} C. Burrin, S. Fairchild and J. Chaika, \emph{Pairs in discrete lattice orbits with applications to Veech surfaces}, preprint {\tt arXiv:2211.14621}
%
\bibitem{EMM98} A. Eskin, G. Margulis and S. Mozes, \emph{Upper bounds and asymptotics in a quantitative version of the Oppenheim conjecture}, Ann. of Math. (2) 147 (1998), 93-141.
%
\bibitem{EMM05} A. Eskin, G. Margulis and S. Mozes, \emph{Quadratic forms of signature (2,2) and eigenvalue spacings on rectangular 2-tori}, Ann. of Math. (2) 161 (2005), no.2, 679–725.
%
\bibitem{Fairchild2021} S. Fairchild, \emph{A higher moment formula for the Siegel--Veech transform over quotients by Hecke triangle groups}, Groups Geom. Dyn. 15 (2021), no. 1, 57-81.
%
\bibitem{FH} S. Fairchild and J. Han, \emph{Mean value theorems for the S-arithmetic primitive Siegel transforms}, preprint {\tt arXiv:2310.03459}
%
\bibitem{GH22} A. Ghosh and J. Han, \emph{Values of inhomogeneous forms at $S$-integral points}, Mathematika 68 (2022), no. 2, 565-593.
%
\bibitem{GKY22} A. Ghosh, D. Kelmer and S. Yu, \emph{Effective density for inhomogeneous quadratic forms I: Generic forms and fixed shifts}, Int. Math. Res. Not. IMRN(2022), no. 6, 4682-4719. 
%
\bibitem{HLM17} J. Han, S. Lim and K. Mallahi-Karai, \emph{Asymtotic distribution of values of isotropic quadratic forms at $S$-integral points}, J. Mod. Dyn. 11 (2017), 501-550.
%
\bibitem{HLM} J. Han, S. Lim and K. Mallahi-Karai, \emph{Asymptotic distribution for pairs of linear and quadratic forms at integral vectors}, Ergod. Th. \& Dynam. Sys. (online published, doi:10.1017/etds.2024.30)
%
\bibitem{Han22} J. Han, \emph{Rogers' mean value theorem for $S$-arithmetic Siegel transforms and applications to the geometry of numbers}, J. Number Theory 240 (2022), 74–106.
%
\bibitem{KS21} D. Kleinbock and M. Skenderi, \emph{Khintchine-type theorems for values of subhomogeneous functions at integer points}, Monatsh. Math. 194 (2021), no. 3, 523-554.
%
\bibitem{KY20} D. Kelmer and S. Yu, \emph{Values of random polynomials in shrinking targets}, Trans. Amer. Math. Soc. 373 (2020), no. 12, 8677-8695.
%
\bibitem{KY21} D. Kelmer and S. Yu, \emph{The second moment of the Siegel transform in the space of symplectice lattices}, Int. Math. Res. Not. IMRN (2021), no. 8, 5825-5859.
%
\bibitem{KY23} D. Kelmer and S. Yu, \emph{Second moment of the light-cone Siegel transform and applications}, Adv. Math. 432 (2023), Paper No. 109270, 70 pp.
%
\bibitem{SKim22} S. Kim, \emph{Mean value formulas on sublattices and flags of the random lattice}, J. Number Theory 241 (2022), 330-351.
%
\bibitem{SKim24} S. Kim, \emph{Adelic Rogers integral formula}, J. Lond. Math. Soc. (2)109 (2024), no.1, Paper No. e12830, 48 pp.
%
\bibitem{MM11} G. Margulis and A. Mohammadi, \emph{Quantitative version of the Oppenheim conjecture for inhomogeneous quadratic forms}, Duke Math. J. 158 (2011), no. 1, 121-160.
%
\bibitem{Rogers55} C. Rogers, \emph{Mean values over the space of lattices}, Acta Math. 94 (1955), 249-287.
%
\bibitem{Rogers55-2} C. Rogers, \emph{The moments of the number of points of a lattice in a bounded set}, Philos. Trans. R. Soc. Lond. Ser. A 248 (1955) 225-251.
%
\bibitem{Rogers56} C. Rogers, \emph{Two integral inequalities}, J. London Math. Soc. 31 (1956), 235-238.
%
\bibitem{Rogers56-2} C. Rogers, \emph{The number of lattice points in a set}, Proc. London Math. Soc. (3) 6 (1956), 305-320.
%
\bibitem{Schmidt57} W. Schmidt, \emph{Mittelwerte über Gitter} (German), Monatsh. Math. 61 (1957), 269-276.
%
\bibitem{Sch60} W. Schmidt, 
\emph{A metrical theorem in geometry of numbers}, 
Trans. Amer. Math. Soc. 95 (1960), 516-529.
%
\bibitem{Sie98} C. Siegel, \emph{A mean value theorem in geometry of numbers}, Ann. Math. 46 (1945),340-347.
\bibitem{Sod2011} A. Södergren,
\emph{On the Poisson distribution of lengths of lattice vectors in a random lattice},
Math. Z. 269 (2011), no. 3-4, 945–954.
\bibitem{StSo2019} A. Strömbergsson and A. Södergren,
\emph{On the generalized circle problem for a random lattice in large dimension}, 
Adv. Math. 345 (2019), 1042–1074.
\end{thebibliography}
\end{document}